\documentclass[final,12pt]{elsarticle}
\usepackage{amsfonts}
\usepackage{epsfig}
\usepackage{makeidx}
\usepackage{amsmath,amssymb,amssymb,latexsym}
\usepackage{enumerate}
\usepackage{graphicx}
\usepackage{mathtools}

\journal{Discrete Mathematics}

\newtheorem{thm}{Theorem}

\newtheorem{lem}[thm]{Lemma}

\newenvironment{proof}{\textit{Proof.}}{\hfill $\Box$ \\}

\renewcommand{\llcorner}{\rule{.4pt}{7pt}\rule{6.6pt}{.4pt}}
\renewcommand{\lrcorner}{\rule{6.6pt}{.4pt}\rule{.4pt}{7pt}}
\renewcommand{\ulcorner}{\rule{.4pt}{7pt}\rule[6.6pt]{6.6pt}{.4pt}}
\renewcommand{\urcorner}{\rule[6.6pt]{6.6pt}{.4pt}\rule{.4pt}{7pt}}
\newcommand{\spac}{\hspace{2pt}}

\date{}
\begin{document}

\begin{frontmatter}

\title{Clique coloring $B_1$-EPG graphs}

\author{Flavia Bonomo}
\address{Universidad de Buenos Aires. Facultad de Ciencias Exactas y Naturales. Departamento de Computaci\'on. Buenos Aires,
Argentina. / CONICET-Universidad de Buenos Aires. Instituto de
Investigaci\'on en Ciencias de la Computaci\'on (ICC). Buenos
Aires, Argentina.} \ead{fbonomo@dc.uba.ar}

\author{Mar\'{\i}a P\'{\i}a Mazzoleni}
\address{CONICET and Departamento de Matem\'atica, FCE-UNLP, La Plata,
Argentina.} \ead{pia@mate.unlp.edu.ar}

\author{Maya Stein}
\address{Departamento de Ingenier\'{\i}a Matem\'atica, Universidad de
Chile, Santiago, Chile.}
\ead{mstein@dim.uchile.cl}

\begin{abstract}
We consider the problem of clique coloring, that is, coloring the
vertices of a given graph such that no (maximal) clique of size at
least two is monocolored. It is known that interval graphs are
$2$-clique colorable. In this paper we prove that $B_1$-EPG graphs
(edge intersection graphs of paths on a grid, where each path has
at most one bend) are $4$-clique colorable. Moreover, given a
$B_1$-EPG representation of a graph, we provide a linear time
algorithm that constructs a $4$-clique coloring of it.
\end{abstract}


\end{frontmatter}

\section{Introduction}

An \textit{EPG representation} $\langle
\mathcal{P},\mathcal{G}\rangle$ of a graph $G$, is a collection of
paths $\mathcal{P}$ of the two-dimensional grid $\mathcal{G}$,
where the paths represent the vertices of $G$ in such a way that
two vertices are adjacent in $G$ if and only if the corresponding
paths share at least one edge of the grid. A graph which has an
EPG representation is called an \textit{EPG graph} (EPG stands for
Edge-intersection of Paths on a Grid). In this paper, we consider
the subclass $B_1$-EPG.
A \textit{B$_1$-EPG representation} of $G$ is an EPG
representation in which each path in the representation has at
most one \textit{bend} (turn on a grid point). Recognizing
$B_1$-EPG graphs is an NP-complete problem~\cite{Heldt}. Also, both the coloring and the maximum independent set problem are NP-complete for $B_1$-EPG graphs~\cite{Epstein}.

EPG graphs have a practical use, for example, in the context of
circuit layout setting, which may be modelled as paths (wires) on
a grid. In the knock-knee layout model, two wires may either cross
or bend (turn) at a common point grid, but are not allowed to
share a grid edge; that is, overlap of wires is not allowed. In
this context, some of the classical optimization graph problems
are relevant, for example, maximum independent set and coloring.
More precisely, the layout of a circuit may have multiple layers,
each of which contains no overlapping paths. Referring to a
corresponding EPG graph, then each layer is an independent set and
a valid partitioning into layers corresponds to a proper coloring.

In this paper, we consider the problem of  \textit{clique
coloring}, that is, coloring the vertices of a given graph such
that no (maximal) clique of size at least two is monocolored.
Clique coloring can be seen also as coloring the clique hypergraph
of a graph. The question of coloring clique hypergraphs was raised
by Duffus et al. in~\cite{Duffus}.

 We prove that $B_1$-EPG graphs are $4$-clique colorable. Moreover,
given a $B_1$-EPG representation of a graph, we provide a linear
time algorithm that constructs a $4$-clique coloring of it.

\section{Preliminaries}
\label{s:preliminares} All graphs considered here are connected,
finite and simple, we follow the notation of~\cite{bondy}. The
vertex set of a graph $G$ is denoted by $V(G)$. A \textit{complete
graph} is a graph that has all possible edges. A \textit{clique}
of a graph $G$ is a maximal complete subgraph of $G$.

A \textit{k-coloring} of a graph $G$ is a function
$f:V(G)\rightarrow \{1,2,\dots,k\}$ such that $f(v)\neq f(w)$ for
adjacent vertices $v,w\in V(G)$. The \textit{chromatic number}
$\chi(G)$ of a graph $G$ is the smallest positive integer $k$ such
that $G$ has a $k$-coloring. A \textit{k-clique coloring} of a
graph $G$ is a function $f:V(G)\rightarrow \{1,2,\dots,k\}$ such
that no clique of $G$ with size at least two is monocolored. A
graph $G$ is \textit{k-clique colorable} if $G$ has a $k$-clique
coloring. The \textit{clique chromatic number} of $G$, denoted by
$\chi_c(G)$, is the smallest $k$ such that $G$ has a $k$-clique
coloring.

Clique coloring has some similarities with usual coloring. For
example, every $k$-coloring is also a $k$-clique coloring,  and
$\chi(G)$ and $\chi_c(G)$ coincide if $G$ is triangle-free. But
there are also essential differences, for example, a clique
coloring of a graph need not be a clique coloring for its
subgraphs. Indeed, subgraphs may have a greater clique chromatic
number than the original graph. Another difference is that even a
$2$-clique colorable graph can contain an arbitrarily large
clique. It is known that the $2$-clique coloring problem is
NP-complete, even under different constraints
\cite{Bacso,Kratochvil}.

Many families of graphs are $3$-clique colorable, for example,
comparability graphs, co-comparability graphs, circular arc graphs
and the $k$-powers of cycles \cite{Campos,Cerioli,Duffus,Duffus2}.
In \cite{Bacso}, Bacs\'o et al.~proved that almost all perfect
graphs are $3$-clique colorable and conjectured that all perfect
graphs are $3$-clique colorable. This conjecture was recently
disproved by Charbit et al.~\cite{Charbit}, who show that there
exist perfect graphs with arbitrarily large clique chromatic
number. Previously known families of graphs having unbounded
clique chromatic number are, for example, triangle-free graphs, UE
graphs (edge intersection graphs of paths in a tree), and line
graphs \cite{Bacso,Cerioli2,Mycielski}.

It has been proved that chordal graphs, and in particular interval
graphs, are $2$-clique colorable~\cite{Poon}. Moreover, the
following result holds for strongly perfect graphs, a superclass
of chordal graphs.

\begin{lem}[Bacs\'o et al.~\cite{Bacso}]\label{l:chordal} Every strongly perfect graph admits a $2$-clique coloring in which one of
the color classes is an independent set.
\end{lem}

For chordal graphs, such a coloring can be easily obtained in
linear time, by a slight modification of the $2$-clique coloring
algorithm for chordal graphs proposed in~\cite{Poon}. Namely, let
$v_1,\dots,v_n$ be a perfect elimination ordering of the vertices
of a chordal graph $G$, i.e., for each $i$, $N[v_i]$ is a clique
of $G[\{v_i,\dots,v_n\}]$; color the vertices from $v_n$ to $v_1$
with colors $a$ and $b$ in such a way that $v_n$ gets color $a$
and $v_i$ gets color $b$ if and only if all of its neighbors that
are already colored got color $a$. A perfect elimination ordering
of the vertices of a chordal graph can be computed in linear
time~\cite{R-T-L-chordal}.

\section{B$_1$-EPG graphs are 4-clique colorable} \label{s:our results}

In this section, we prove that $B_1$-EPG graphs are $4$-clique
colorable. We need the following definitions and theorem.

Let $\langle \mathcal{P},\mathcal{G}\rangle$ be a $B_1$-EPG
representation of a graph $G$. A clique $C$ of $G$ is an
\textit{edge clique} of $\langle \mathcal{P},\mathcal{G}\rangle$
if all the paths of $\mathcal{P}$ that correspond to the vertices
of $C$ share a common edge of the grid $\mathcal{G}$. A clique $C$
of $G$ is a \textit{claw clique} of $\langle
\mathcal{P},\mathcal{G}\rangle$ if there is a point $x$ of the
grid and three edges of the grid sharing $x$ (they may be shaped
$\bot$, $\top$, $\vdash$, or $\dashv$), such that each path of
$\mathcal{P}$ that corresponds to a vertex of $C$ contains two of
these three edges, and every pair of these three edges is
contained in at least one path $P$ of $\mathcal{P}$ (so, it is not
an edge clique). We say that the claw clique is \textit{centered
at} $x$, or that $x$ is the \emph{center} of the claw clique. An
example can be seen in Figure~\ref{cliques}.

\begin{figure}[h] \centering{
\includegraphics[width=.5\textwidth]{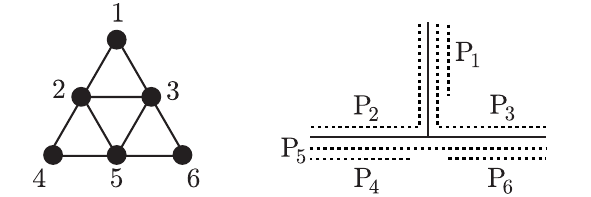}
\caption{A $B_1$-EPG representation of the 3-sun. The central
triangle $\{2, 3, 5\}$ is a claw clique; the other three triangles
are edge cliques (figure from \cite{Golumbic}).}\label{cliques}}
\end{figure}

\begin{thm}[Golumbic et al.~\cite{Golumbic}]\label{t:1} Let $\langle \mathcal{P},\mathcal{G}\rangle$ be a $B_1$-EPG representation of a graph $G$.
Every clique in $G$ corresponds to either an edge clique or a claw
clique in $\langle \mathcal{P},\mathcal{G}\rangle$.\end{thm}

Now we can prove the main result of this paper.

\begin{thm} Let $G$ be a $B_1$-EPG graph. Then, $G$ is 4-clique colorable. Moreover, given a $B_1$-EPG representation of $G$, a 4-clique coloring of $G$ in which one of
the color classes is an independent set can be obtained in linear
time on the number of vertices and edges. \end{thm}

\begin{proof}
Let $\langle \mathcal{P},\mathcal{G}\rangle$ be a $B_1$-EPG
representation of the graph $G$. Each path of $\mathcal{P}$ is
composed of either a single segment, formed by one or more edges
on the same row or column of the grid $\mathcal{G}$, or of two
segments sharing a point of the grid, one horizontal (i.e., in a
row) and one vertical (i.e., in a column). We will first assign
colors independently to the horizontal and vertical segments of
each path, and then we will show how to combine those colors into
a single color for each path, as required.

First, we use Lemma~\ref{l:chordal} to color the segments on each row and each column of
$\mathcal{G}$ as if they were vertices of an interval graph, with
two colors $a$ and $b$, such that the segments colored~$b$ form an
independent set, i.e., a pairwise non intersecting set, on each
row (respectively column).

We thus obtain four types of paths according to the colors given to
their corresponding segments: $(a,a)$, $(a,b)$, $(b,a)$, $(b,b)$,
where the first component corresponds to the horizontal segment of
the path and the second component corresponds to the vertical segment of the
path; if one of these parts does not exist, we assign an $a$ to the
missing component.

Observe that
\begin{equation}\label{bbindependent}
\text{the color class $(b,b)$ is an independent set.}
\end{equation}

Let us now investigate which cliques could be monocolored.
Edge cliques of $\langle \mathcal{P},\mathcal{G}\rangle$ are also
cliques of the interval graph corresponding to the row of the
grid, respectively column of the grid, to which the edge (where
all the paths of the clique intersect) belongs. Thus, the colors
of the paths in such a clique have to be different in the
horizontal component (respectively in the vertical component),
that is, the clique is not monocolored.

Let us now turn to the claw cliques. Suppose that there is a claw
clique which is monocolored. Then this clique contains at least
two paths whose horizontal  segments overlap, and have the same
color, and the same is true for vertical segments.  Since  the
horizontal  segments  colored $b$ form an independent set, and the
same is true for the vertical segments, the only possible coloring
of the paths in our monocolored claw clique is $(a,a)$.

Now,  for each point $x$ of the grid which is the center of
one or more claw cliques monocolored $(a,a)$, we will perform a
recoloring of one or two paths having a bend at $x$. In this way,
each path will be recolored at most once, as it has at most one
bend. Paths without bends will not be recolored.

The order in which we process the points $x$ of the grid does not matter: The recolorings are independent of recolorings at other grid points. In the recoloring we will assign color $b$ to some segments
that were originally colored $a$, obeying the following rules, for any fixed point $x$ of the grid:

\begin{enumerate}[(I)]
\item the recolored paths either get color $(a,b)$ or $(b,a)$,\label{I}
\item every segment of a path with a bend at $x$ that is
recolored $b$ is contained in a segment of a path with a bend
at $x$ that is colored $a$; \label{Ix}
\item if we recolor two paths with a bend at $x$, they only share $x$ (i.e., they are shaped $\llcorner$ and $\urcorner$ or $\ulcorner$ and $\lrcorner$); \label{II}
\item \label{III} after recoloring, there is no claw clique colored $(a,a)$ centered at $x$.
\end{enumerate}

It will be explained below how to construct a recoloring obeying
rules (\ref{I}) -- (\ref{III}). Once such a recoloring is found,
the segments colored $b$ may no longer be an independent set, but
properties~(\ref{I}) -- (\ref{II}) prevent us from creating new
monocolored cliques. Indeed, property~(\ref{Ix}) ensures that we
create no monocolored edge clique. We claim that
properties~(\ref{I}) --~(\ref{II}) guarantee we have no new
monocolored claw clique. Assume instead that a claw clique
centered at a grid point $x$ gets monocolored after the process,
and by symmetry assume it is shaped $\bot$. By (\ref{I}), it is
either monocolored $(a,b)$ or $(b,a)$. In the first case, the
vertical segment of one of the paths having a bend at $x$, let us
say $P$, has to be recolored, because the segments that were
originally colored $b$ formed an independent set. Property
(\ref{Ix}) implies that there is a path $Q$ having a bend at $x$
and whose vertical segment contains the vertical segment of $P$
and is colored $a$; This leads to a contradiction, because by
maximality, $Q$ belongs to the clique. In the second case, since
by (\ref{II}) at most one of the paths of the clique that have a
bend at $x$ was recolored, and the segments that were originally
colored $b$ formed an independent set, the horizontal segments of
all the paths that belong to the clique and do not have a bend at
$x$ were recolored $b$. Property (\ref{Ix}) implies that there is
a path belonging to the clique whose horizontal segment is colored
$a$, a contradiction as well. These observations and
property~\eqref{III} ensure that after going through all grid
points, we have found a $4$-clique coloring of $G$. By
\eqref{bbindependent} and by Property~\eqref{I}, the coloring has
an independent color class.

\smallskip

Let us now explain how we find the recoloring with properties
(\ref{I}), \eqref{Ix}, (\ref{II}) and (\ref{III}), for a fixed
grid point $x$. We distinguish three cases. Let us say a shape is
{\it missing} at $x$ if either there is no path of this shape with
a bend at $x$ or there is at least one path of this shape with a
bend at $x$ that is not colored $(a,a)$.
\medskip

\textbf{Case 1:}
Two or more of the shapes $\lrcorner$\spac,  $\llcorner$\spac, $\urcorner$\spac, $\ulcorner$  are missing at $x$.

\smallskip
If there is no $(a,a)$-colored claw clique centered at $x$, we do not recolor anything. Clearly,~\eqref{I} --~\eqref{III} hold.
Otherwise, there is a unique $(a,a)$-colored claw clique at $x$, and symmetry allows us to assume this clique is shaped $\bot$.
Both shapes $\urcorner$  and  $\ulcorner$ are missing at $x$.
Of all $\lrcorner\spac$- or $\llcorner\spac$-shaped paths with bend at~$x$, choose the one with the shortest vertical segment, and recolor it $(a,b)$. Then~\eqref{I} --~\eqref{III} hold.

\medskip

\textbf{Case 2:}
Exactly one of the shapes $\lrcorner$\spac,  $\llcorner$\spac, $\urcorner$\spac, $\ulcorner$  is missing at $x$.

\smallskip
By symmetry, we can assume $\ulcorner$  is missing at $x$.
Let $\mathcal P$ be the set of  all paths with bend at $x$ that have the shape $\lrcorner$.
If there is a path $P\in \mathcal P$ whose horizontal segment is contained in another path with bend at $x$, then recolor $P$ with $(b,a)$. Otherwise, if there is a path $P$ in $\mathcal P$ whose vertical segment is contained in another path with bend at $x$, then recolor $P$ with $(a,b)$. In both cases, the choice of $P$ ensures that~\eqref{I} --~\eqref{III} hold.

It remains to consider the case that for each of the paths in $\mathcal P$, their horizontal (vertical) segment strictly contains all horizontal (vertical) segments of paths with bend at $x$ (in particular, $|\mathcal P|=1$). Then, choose any $\llcorner\spac$-shaped path $P_1$  and any $\urcorner\spac$-shaped  path $P_2$ with bend at $x$, recolor $P_1$ with $(a,b)$ and $P_2$ with $(b,a)$ and observe  that~\eqref{I} --~\eqref{III} hold by the choice of $P_1$ and $P_2$.
\medskip

\textbf{Case 3:}
None of the shapes $\lrcorner$\spac,  $\llcorner$\spac, $\urcorner$\spac, $\ulcorner$  is missing at $x$.

\smallskip
Consider the shortest of all segments of paths with bend at $x$ (or one of them if there is more than one), and let $Q$ be the path it belongs to. By symmetry, we may assume~$Q$ is shaped $\ulcorner$, and the shortest segment is the horizontal segment. As in the previous case, let
$\mathcal P$ be the set of  all $\lrcorner\spac$-shaped paths with bend at $x$. If there is a path $P\in\mathcal P$ whose horizontal (vertical) segment is contained in another path with bend at $x$, then recolor $P$ with $(b,a)$ (or with $(a,b)$, respectively), and recolor $Q$ with $(b,a)$. The choice of $P$ and $Q$ guarantees~\eqref{I}--~\eqref{III}.

Otherwise, for each of the paths in $\mathcal P$, their horizontal (vertical) segment strictly contains all horizontal (vertical) segments of paths with bend at $x$. Choose any $\llcorner\spac$-shaped path $P_1$  and any $\urcorner\spac$-shaped  path $P_2$ with bend at $x$, and recolor $P_1$ with $(a,b)$ and $P_2$ with $(b,a)$ ($Q$ is not recolored in this case). Again,~\eqref{I} --~\eqref{III} hold.

\medskip

The presented algorithm gives a $4$-clique coloring of $G$, where
one color class is an independent set. The algorithm can be implemented
to run in linear time in the number of vertices and edges of $G$.
\end{proof}

\section{Conclusion}

In this paper we have proved that $B_1$-EPG graphs are $4$-clique
colorable, and that such coloring can be obtained in linear time
in the number of vertices and edges of the graph, given a
$B_1$-EPG representation of it. This algorithm may use four colors
in graphs that are in fact 2-clique colorable or 3-clique
colorable.

We conjecture that indeed $B_1$-EPG graphs are $3$-clique
colorable. Examples of $B_1$-EPG graphs that require three colors
for a clique coloring are the odd chordless cycles, and we could
not find examples of $B_1$-EPG graphs having clique chromatic
number 4 (The Mycielski graph with chromatic number 4, line graphs
of big cliques and the graph in \cite{Charbit} having clique
chromatic number 4 are not $B_1$-EPG).

Further open questions are the computational complexity of
2-clique coloring and 3-clique coloring on $B_1$-EPG graphs. \\

\bigskip

\noindent
\textbf{Acknowledgements:} This work was partially supported by UBACyT Grant
    20020130100808BA, CONICET PIP 122-01001-00310 and
    112-201201-00450CO, and ANPCyT PICT 2012-1324 and 2015-2218 (Argentina),
Fondecyt grant 1140766 and ``Nucleo Milenio Informaci\'on y
Coordinaci\'on en Redes'' ICM/FIC P10-024F (Chile), and MathAmSud
Project 13MATH-07 (Argentina--Brazil--Chile--France).


\begin{thebibliography}{10}
\expandafter\ifx\csname url\endcsname\relax
  \def\url#1{\texttt{#1}}\fi
\expandafter\ifx\csname
urlprefix\endcsname\relax\def\urlprefix{URL }\fi
\newcommand{\enquote}[1]{``#1''}

\bibitem{Bacso} G. Bacs\'{o}, S. Gravier, A. Gy\'{a}rf\'{a}s, M. Preissmann and  A. Seb\H{o},
\emph{Coloring the maximal cliques of graphs}, SIAM Journal on
Discrete Mathematics. 17 (2004) 361--376.

\bibitem{bondy} J. Bondy and U. Murty, \emph{Graph Theory}, Springer, New York. (2007).

\bibitem{Campos} C. N. Campos, S. Dantas and  C. P. de Mello,
\emph{Colouring clique-hypergraphs of circulant graphs}, Graphs
and Combinatorics. 29(6) (2013) 1713--1720.

\bibitem{Cerioli} M. R. Cerioli and  A. L. Korenchendler,
\emph{Clique coloring circular-arc graphs}, Electronic Notes in
Discrete Mathematics. 35 (2009) 287--292.

\bibitem{Cerioli2} M. R. Cerioli and  P. Petito,
\emph{Clique coloring UE and UEH graphs}, Electronic Notes in
Discrete Mathematics. 30 (2008) 201--206.

\bibitem{Charbit} P. Charbit, I. Penev, S. Thomass\'e and N.
Trotignon, \emph{Perfect graphs of arbitrarily large
clique-chromatic number}, Journal of Combinatorial Theory, Series
B. 116 (2016) 456--464.


\bibitem{Duffus} D. Duffus, B. Sands, N. Sauer and  R. E. Woodrow,
\emph{Two-colouring all two-element maximal antichains}, Journal
of Combinatorial Theory Series A. 57 (1991) 109--116.

\bibitem{Duffus2} D. Duffus, H. A. Kierstead and  W. T. Trotter,
\emph{Fibres and ordered set coloring}, Journal of Combinatorial
Theory Series A. 58 (1991) 158--164.


\bibitem{Epstein} D. Epstein, M. C. Golumbic, and G. Morgenstern,
\emph{Approximation Algorithms for $B_1$-{EPG} graphs},
In: {A}lgorithms and {D}ata Structures, WADS 2013 Proceedings, Vol. 8037 of Lecture Notes in Computer Science, Springer. (2013) 328--340.

\bibitem{Golumbic} M. C. Golumbic, M. Lipshteyn and M. Stern,
\emph{Edge intersection graphs of single bend paths on a grid},
Networks. 54 (2009) 130--138.

\bibitem{Heldt} D. Heldt, K. Knauer and  T. Ueckerdt,
\emph{Edge-intersection graphs of grid paths: the bend number},
Discrete Applied Mathematics. 167 (2014) 144--162.

\bibitem{Kratochvil} J. Kratochv\'{\i}l and Zs. Tuza, \emph{On the complexity of bicoloring
clique hypergraphs of graphs}, in Proceedings of the 11th Annual
ACM-SIAM Symposium on Discrete Algorithms, SIAM, Philadelphia, PA.
(2000) 40--41.

\bibitem{Mycielski} J. Mycielski,
\emph{Sur le coloriage des graphes}, Colloquium Mathematicum. 3
(1955) 161--162.

\bibitem{Poon} H. Poon, \emph{Coloring Clique Hypergraphs}, Master's thesis, West
Virginia University (2000).

\bibitem{R-T-L-chordal} D. {Rose}, R. {Tarjan} and G. {Lueker},
\emph{Algorithmic aspects of vertex elimination on graphs}, SIAM
Journal on Computing. 5 (1976) 266--283.

\end{thebibliography}
\end{document}